\documentclass{amsart}
\usepackage[utf8]{inputenc}
\usepackage{todonotes} 
\usepackage{amsmath, amssymb, amsthm} 
\usepackage{latexsym} 
\usepackage{dsfont}
\usepackage[flushleft]{paralist}
\usepackage[all]{xy}
\usepackage{tikz-cd}
\usepackage{stmaryrd} 
\usepackage{hyperref} 
\usepackage[OT2,T1]{fontenc}
\tikzcdset{scale cd/.style={every label/.append style={scale=#1},
    cells={nodes={scale=#1}}}}
\DeclareSymbolFont{cyrletters}{OT2}{wncyr}{m}{n}
\DeclareMathSymbol{\Zhe}{\mathalpha}{cyrletters}{"11} 
    
\begin{document}
\title{Fine Selmer groups of modular forms}
\author{Sören Kleine} 
\address{Institut für Theoretische Informatik, Mathematik und Operations Research, Universität der Bundeswehr München, Werner-Heisenberg-Weg 39, 85577 Neubiberg, Germany} 
\email{soeren.kleine@unibw.de} 
\author{Katharina Müller} 
\address{Department of Mathematics and Statistics, Université Laval,  Québec City, Canada} 
\email{katharina.mueller.1@ulaval.ca} 


\newcommand{\R}{\mathds{R}}
 	\newcommand{\Z}{\mathds{Z}}
 	\newcommand{\N}{\mathds{N}}
 	\newcommand{\Q}{\mathds{Q}}
 	\newcommand{\K}{\mathds{K}}
 	\newcommand{\M}{\mathds{M}} 
 	\newcommand{\C}{\mathds{C}}
 	\newcommand{\B}{\mathds{B}}
 	\newcommand{\LL}{\mathds{L}}
 	\newcommand{\F}{\mathds{F}}
 	\newcommand{\p}{\mathfrak{p}} 
 	\newcommand{\q}{\mathfrak{q}} 
 	\newcommand{\f}{\mathfrak{f}} 
 	\newcommand{\Pot}{\mathcal{P}}
 	\newcommand{\Gal}{\textup{Gal}}
 	\newcommand{\rg}{\textup{rank}}
 	\newcommand{\id}{\textup{id}}
 	\newcommand{\Ker}{\textup{Ker}}
 	\newcommand{\Image}{\textup{Im}} 
 	\newcommand{\pr}{\textup{pr}}
 	\newcommand{\la}{\langle}
 	\newcommand{\ra}{\rangle}
 	\newcommand{\gdw}{\Leftrightarrow}
 	\newcommand{\pfrac}[2]{\genfrac{(}{)}{}{}{#1}{#2}}
 	\newcommand{\Ok}{\mathcal{O}}
 	\newcommand{\Norm}{\mathrm{N}} 
 	\newcommand{\coker}{\mathrm{coker}}
 	\newcommand{\dotcup}{\stackrel{\textstyle .}{\bigcup}}
 	\newcommand{\Cl}{\mathrm{Cl}}
 	\newcommand{\Sel}{\textup{Sel}} 
 	\newcommand{\OkG}{\Ok\llbracket G\rrbracket }

 	\newtheorem{lemma}{Lemma}[section] 
 	\newtheorem{prop}[lemma]{Proposition} 
 	\newtheorem{defprop}[lemma]{Definition and Proposition} 
 	\newtheorem{conjecture}[lemma]{Conjecture} 
 	\newtheorem{thm}[lemma]{Theorem} 
 	\newtheorem*{thm*}{Theorem} 
 	\newtheorem{cor}[lemma]{Corollary}
 	\newtheorem{claim}[lemma]{Claim}

 	\theoremstyle{definition}
 	\newtheorem{def1}[lemma]{Definition} 
 	\newtheorem{ass}[lemma]{Assumption}
 	\newtheorem{rem}[lemma]{Remark} 
 	\newtheorem{rems}[lemma]{Remarks} 
 	\newtheorem{example}[lemma]{Example} 
 	\newtheorem{fact}[lemma]{Fact}

\maketitle

\begin{abstract}  
  We compare the Iwasawa invariants of fine Selmer groups of $p$-adic Galois representations over admissible $p$-adic Lie extensions of a number field $K$ to the Iwasawa invariants of ideal class groups along these Lie extensions. 
  
  More precisely, let $K$ be a number field, let $V$ be a $p$-adic representation of the absolute Galois group $G_K$ of $K$, and choose a $G_K$-invariant lattice ${T \subseteq V}$. We study the fine Selmer groups of ${A = V/T}$ over suitable $p$-adic Lie extensions $K_\infty/K$, comparing their corank and  $\mu$-invariant to the corank and the $\mu$-invariant of the Iwasawa module of ideal class groups in $K_\infty/K$. 
  
  In the second part of the article, we compare the Iwasawa $\mu$- and $l_0$-invariants of the fine Selmer groups of CM modular forms on the one hand and the Iwasawa invariants of ideal class groups on the other hand over trivialising multiple $\Z_p$-extensions of $K$. 
\end{abstract}

\section{Introduction} \label{section:intro} 
Almost 20 years ago, Coates-Sujatha, Wuthrich et. al. (see \cite{coates-sujatha}, \cite{wuthrich-fine-Selmer}) introduced the study of \emph{fine Selmer groups} into modern Iwasawa theory. The corresponding Iwasawa modules are certain quotients of the Pontryagin duals of the ordinary Selmer groups, which satisfy more natural finiteness conditions (i.e. whereas the Pontryagin duals of the Selmer groups are often non-torsion over the Iwasawa algebras, the Pontryagin duals of the fine Selmer groups are expected to be always torsion, and they indeed are in all known cases). 
One further motivation for the investigation of fine Selmer groups was the discovery of deep analogies between the Iwasawa modules of fine Selmer groups and ideal class groups in the seminal paper \cite{coates-sujatha} of Coates and Sujatha. For the purpose of illustration, we mention just one example of such an analogy. In \cite[Conjecture~A]{coates-sujatha}, Coates and Sujatha conjectured that the Iwasawa $\mu$-invariant of the fine Selmer group of an elliptic curve over the cyclotomic $\Z_p$-extension of a number field $K$ always vanishes. This is the analogue of a very classical conjecture of Iwasawa about the growth of ideal class groups along the cyclotomic $\Z_p$-extension of a number field $K$. 

Coates and Sujatha have proved in \cite[Theorem~3.4]{coates-sujatha} that the two conjectures are equivalent if the number field $K$ contains the group $E[p]$ of $p$-torsion points on the elliptic curve $E$. This has been generalised to the context of abelian varieties and non-cyclotomic $\Z_p$-extensions by Lim and Murty (see \cite[Theorem~5.5]{lim-murty}. In \cite{kleine-mueller2}, we generalised this analogy further. In order to describe the results from \cite{kleine-mueller2}, we have to introduce more notation. 

Fix a prime $p$ and a number field $K$. If ${p = 2}$, then we assume $K$ to be totally imaginary. Let $K_\infty/K$ be a uniform $p$-adic Lie extension (see Section~\ref{section:notation} for the definitions) with Galois group $G$, and let $A$ be an abelian variety which is defined over $K$. We choose a finite set $\Sigma$ of primes of $K$ which contains the set $\Sigma_p$ of primes above $p$, the set $\Sigma_{\textup{br}}(A)$ of primes where $A$ has bad reduction and the set $\Sigma_{\textup{ram}}(K_\infty/K)$ of primes which ramify in $K_\infty/K$. We denote by $Y_{A, \Sigma}^{(K_\infty)}$ the Pontryagin dual of the $\Sigma$-fine Selmer group of $A$ over $K_\infty$, and we let ${Y(K_\infty) = \varprojlim_{K'}(Y(K'))}$ be the projective limit of the $p$-primary subgroups of the ideal class groups of the finite subextensions ${K' \subseteq K_\infty}$ of $K$. 

The following statement is \cite[Theorem~1.1]{kleine-mueller2}: suppose that no prime ${v \in \Sigma}$ splits completely in $K_\infty/K$ and that ${A[p^k] \subseteq A(K_\infty)}$ for some ${k \in \N}$. Then 
\[ j \cdot \rg_{\Z_p\llbracket G\rrbracket }(Y_{A, \Sigma}^{(K_\infty)}) + \sum_{i = 0}^j f_i(Y_{A, \Sigma}^{(K_\infty)}) =  2d \cdot \sum_{i = 0}^j f_i(Y(K_\infty)) \] 
for each $j \le k$, where 
\[ f_i(X) = \rg_{\F_p\llbracket G\rrbracket }(p^i X[p^\infty]/p^{i+1} X[p^\infty])\] 
for any finitely generated $\Z_p\llbracket G \rrbracket $-module $X$, and where $d$ denotes the dimension of $A$. 

The first goal of the present article is to generalise this result to a context which is as broad as possible. To this purpose, let $K_\infty/K$ be an arbitrary admissible Lie extension (see Section~\ref{section:notation}), and let $V$ be a $p$-adic representation of the absolute Galois group $G_K$ of $K$, defined over a finite extension $F$ of $\Q_p$. Choose a $G_K$-invariant lattice ${T \subseteq V}$, and let ${A = V/T}$. Finally, let $\Ok$ be the ring of integers of $F$, choose some uniformising element $\pi$ of $\Ok$, and let $\F_q$ be the residue field $\Ok/(\pi)$. Then the fine Selmer group of $A$ over $K_\infty$ is a finitely generated $\OkG$-module, where ${G = \Gal(K_\infty/K)}$ as above. In this context, we prove the following result. 
\begin{thm} \label{thm:A} 
  Let $A = V/T$ be as above, and let $K_\infty/K$ be an admissible $p$-adic Lie extension with Galois group $G$. We fix a finite set $\Sigma$ which contains $\Sigma_p$, $\Sigma_{\textup{ram}}(A)$ and $\Sigma_{\textup{ram}}(K_\infty/K)$. If ${p = 2}$, then we assume that $K$ is totally imaginary. 
  
  Let ${k \in \N}$. If the decomposition subgroup ${D_v \subseteq G}$ is infinite for each ${v \in \Sigma}$ and ${A[\pi^k] \subseteq A(K_\infty)}$, then
  \begin{align} \label{eq:thmA} j \cdot \rg_{\OkG}(Y_{A, \Sigma}^{(K_\infty)}) + \sum_{i = 0}^j f_i(Y_{A, \Sigma}^{(K_\infty)}) = d \cdot \sum_{i = 0}^j f_i(Y(K_\infty)\otimes_{\Z_p} \Ok) \end{align} 
  for each $j \le k$, where now 
  \[ f_i(X) = \rg_{\F_q \llbracket G\rrbracket }(\pi^i X[p^\infty]/\pi^{i+1} X[p^\infty]) \] 
  for any finitely generated $\OkG$-module $X$, and where ${d = \dim(A)}$. 
\end{thm} 
In order to prove this result, we use an approach which originally goes back to Coates and Sujatha (see \cite[Lemma~3.8]{coates-sujatha}) and has been exploited in many other papers (for example, in \cite{lim-notesonthefineSelmergroups}, \cite{kleine-mueller2}, \cite{lim-murty}). These results relate the ideal class group over some number field containing a sufficiently large part of the $\pi$-torsion of an abelian variety $A$ to certain truncated fine Selmer groups of $A$ (see Section~\ref{section:fine-Selmer-groups} for the definitions) and will be generalised to our context in Lemma~\ref{lemma_lim-murty}. 
Then we use a control theorem in order to derive information about the structural invariants of $Y_{A, \Sigma}^{(K_\infty)}$. 

In the second part of the article, we consider Rankin-Selberg convolutions of CM modular forms $f_1, \ldots, f_k$ (see Definition~\ref{def:CM-Form}) over trivialising multiple $\Z_p$-extensions. In this setting, we can compare also finer structural invariants like the so-called \emph{$l_0$-invariants}, which are generalised Iwasawa-$\lambda$-invariants and have been introduced first by Cuoco and Monsky in \cite{cuoco-monsky} in the context of ideal class group Iwasawa modules. To this purpose, we need to assume that a certain finiteness condition $(\star)$ holds. This is an assumption the occurrence which is quite natural in this context (see for example Hypothesis~$H_{\textup{cyc}}$ in \cite{hatley-lei-kundu-ray}, or Assumption~(H0) in \cite{lei-lim}). We prove the following 
\begin{thm} \label{thm:intro-B} 
  Let $f_1, \ldots, f_k$ be CM modular forms, and let ${M = A_{f_1} \otimes \ldots \otimes A_{f_k}}$ be associated with the Rankin-Selberg convolution of ${f_1, \ldots, f_k}$. 
  
  We assume that $M[p] \subseteq M(K)$, and we let $K_\infty$ be a trivialising extension for $M$, i.e. $K_\infty$ is a $\Z_p^l$-extension of a finite extension $K'$ of $K$ such that $M$ is defined over $K_\infty$. If ${k > 1}$, then we assume that 
  \[ (\star): \quad \textup{$M^{D_\p}$ is finite for each prime $\p$ of $K_\infty$ above $p$,} \]  
  where ${D_\p \subseteq \Gal(K_\infty/K')}$ denotes the decomposition subgroup. 
  
  Then 
  \[ \mu_{\OkG}(Y_M^{(K_\infty)}) = 2^k ef \mu_{\OkG}(Y(K_\infty) \otimes_{\Z_p} \Ok), \] 
  \[ (l_0)_{\OkG}(Y_M^{(K_\infty)}) = 2^k e f (l_0)_{\OkG}(Y(K_\infty) \otimes_{\Z_p} \Ok), \] 
  and $\rg_{\OkG}(Y_M^{(K_\infty)}) = 0$. Here $e$ and $f$ denote the ramification index and the inertia degree of the prime $p$ in the local extension $F/\Q_p$, and $Y_M^{(K_\infty)}$ denotes the Pontryagin dual of the fine Selmer of $M$ over $K_\infty$. 
\end{thm} 
The fine Selmer groups in Theorem~\ref{thm:intro-B} do not depend on the choice of the finite set $\Sigma$, as long as $\Sigma$ contains $\Sigma_p$ and the sets $\Sigma_{\textup{ram}}(M)$ and $\Sigma_{ram}(K_\infty/K)$ of primes at which $M$ ramifies, respectively, which ramify in the extension $K_\infty/K$ (see Lemma~\ref{lemma:sujatha-witte}). In the proof of Theorem~\ref{thm:intro-B}, we apply deep results from \cite{kleine-mueller2} in order to relate the growth of the ideal class groups of the intermediate number fields to the structural invariants of $Y(K_\infty)$

Let us briefly describe the structure of the article. In Section~\ref{section:notation}, besides introducing the basic setup and the notations, we prove several auxiliary results comparing the $\OkG$- and $\Z_p\llbracket G \rrbracket $-Iwasawa invariants of general Iwasawa modules. Since most authors study only the $\Z_p\llbracket G\rrbracket $-module structure of Iwasawa modules, this part of the article (although in parts a bit technical) might be of independent interest. In Section~\ref{section:auxiliary-results}, we prove several further auxiliary results. Section~\ref{section:control} is devoted to the proof of the control theorems which are needed for the proofs of the main results. These latter proofs are contained in the final section of the article. 

\textbf{Acknowledgements.} The second named author was supported by Antonio Lei's NSERC Discovery Grants RGPIN-2020-04259 and RGPAS-2020-00096. We would like to thank Antonio Lei 
for fruitful discussions during the preparation of this paper. 

\section{Background and basic definitions} \label{section:notation} 
\subsection{General notation} \label{subsection:general} 
Let $p$ be a rational prime. Throughout the article we fix a finite extension $F$ of $\Q_p$ with ring of integers $\Ok$. Let ${\pi \in \Ok}$ be a uniformising element. 

For any Noetherian $\Ok$-module $M$, we denote by $M[\pi^\infty]$ the subset of $\pi$-power torsion elements. For any ${i \in \N}$, ${i > 0}$, we denote by $M[\pi^i]$ the subset of $M[\pi^\infty]$ which consists exactly of the elements ${x \in M}$ which are annihilated by $\pi^i$. 

Similary, for any group $G$, $G[p^\infty]$ shall denote the subset of elements of $p$-power order. 

Fix a number field $K$ and a finite set $\Sigma = \Sigma(K)$ of primes of $K$. For any (non-necessarily finite) algebraic extension $L$ of $K$, we write $\Sigma(L)$ for the set of primes of $L$ which divide some ${v \in \Sigma}$. Moreover, $\Sigma_p(L)$ shall denote the set of primes of $L$ lying above $p$. In the case of ${L = K}$, we abbreviate $\Sigma_p(K)$ to $\Sigma_p$. 

Let $V$ be a finite-dimensional $F$-vector space with a continuous action of the absolute Galois group ${G_K = \Gal(\overline{K}/K)}$ of $K$, for some fixed algebraic closure $\overline{K}$ of $K$. Fix a Galois stable $\Ok$-lattice $T$ in $V$, and write ${A = V/T}$. Then $A$ is isomorphic as an $\Ok$-module to $(F/\Ok)^d$ for some non-negative integer $d$ which we call the \emph{dimension} of the Galois representation $V$. In particular, $A$ is $\pi$-primary, i.e. ${A = A[\pi^\infty]}$. By abuse of terminology, we will also refer to $d$ as the dimension of $A$. 

If $V$ is a $p$-adic Galois representation of $\Gal(\overline{K}/K)$ and ${A = V/T}$, as above, then we denote by $\Sigma_{\textup{ram}}(A)$ the set of primes of $K$ where $V$ is ramified. Note: in this article, we consider only Galois representations with a finite ramification set. 

\subsection{Admissible $p$-adic Lie extensions, Iwasawa modules and $\mu$-invariants} \label{subsection:Iwasawa-modules} 
\begin{def1} 
   An \emph{admissible $p$-adic Lie extension} $K_\infty$ of a number field $K$ shall be a normal extension $K_\infty/K$ such that \begin{compactitem} 
     \item ${G = \Gal(K_\infty/K)}$ is a compact pro-$p$, $p$-adic Lie group, 
     \item ${G[p^\infty] = \{0\}}$, and 
     \item the set $\Sigma_{\textup{ram}}(K_\infty/K)$ of primes of $K$ which ramify in $K_\infty$ is finite. 
  \end{compactitem} 
\end{def1} 

If $K_\infty/K$ is an admissible $p$-adic Lie extension with Galois group $G$, then the completed group ring 
$$\Ok\llbracket G \rrbracket = \varprojlim_{U\subseteq G} \Ok[G/U]$$ 
is a Noetherian domain by \cite[Theorem~2.3]{coates-howson} (here $U$ runs over the open normal subgroups of $G$). For any finitely generated $\Ok\llbracket G \rrbracket $-module $M$, we define the $\OkG$-rank of $M$ as 
\[ \rg_{\OkG}(M) = \dim_{\mathcal{F}(G)}(\mathcal{F}(G) \otimes_{\OkG} M), \] 
where $\mathcal{F}(G)$ denotes the skew field of fractions of $\OkG$ (the field of fractions exists by \cite[Corollary~.25]{Dixon} and \cite[Chapter~10]{goodearl-warfield}). 
A finitely generated $\OkG$-module $M$ is called \emph{torsion} if ${\rg_{\OkG}(M) = 0}$, and a torsion $\OkG$-module $M$ is called \emph{pseudo-null} if ${\textup{Ext}^1_{\OkG} (M, \OkG) = \{0\}}$ (this follows the definitions given in \cite{venjakob} and \cite{coates-schneider-sujatha}). Two finitely generated $\OkG$-modules $M$ and $M'$ are called \emph{pseudo-isomorphic} if there exists a $\OkG$-module homomorphism ${\varphi \colon M \longrightarrow M'}$ with pseudo-null kernel and cokernel. 

Following Howson (see \cite[formula~(33)]{howson}), we define the \emph{$\mu$-invariant} of a finitely generated $\OkG$-module $M$ as 
\begin{align} \label{eq:def-mu} \mu(M) = \mu_{\OkG}(M) = \sum_{i \, \ge \, 0} \rg_{\F_q \llbracket G\rrbracket } (\pi^i M[\pi^\infty]/\pi^{i+1} M[\pi^\infty]), \end{align}  
where ${\F_q = F/\Ok}$ is a finite field with $p^f$ elements (here ${f \in \N}$ denotes the inertia degree of $p$ in the extension $F/\Q_p$). 

Now let $M$ be a finitely generated $\OkG$-module. By \cite[Theorem~3.40]{venjakob}, $M[\pi^\infty]$ is pseudo-isomorphic to a direct sum 
\[ \bigoplus_{i = 1}^s \OkG/(\pi^{e_i}), \] 
where $e_1, \ldots, e_s \in \N$ and 
\[ e_1 + \ldots + e_s = \mu_{\OkG}(M). \] 
In this setting, we define the \emph{$k$-truncated $\mu$-invariant of $M$}, ${k \in \N}$, as 
\[ \mu^{(k)}(M) = \mu^{(k)}_{\OkG}(M) = \sum_{i = 1}^s \min(e_i,k) . \] 
It is easy to see that for a finitely generated torsion $\OkG$-module $M$, we have 
\[ \mu^{(k)}(M) = \mu(M[\pi^\infty]/\pi^k M[\pi^\infty]) = \sum_{i=0}^k \rg_{\F_q\llbracket G\rrbracket }(\pi^i M[\pi^\infty]/\pi^{i+1} M[\pi^\infty]), \] 
see also \cite[Lemma~5.8]{kleine-mueller2}. 

Modules of the form $\bigoplus_{i = 1}^s \OkG/(\pi^{e_i})$ are so-called \emph{elementary $p$-torsion $\OkG$-modules}.

Every $\Ok\llbracket G \rrbracket $-module is also a $\Z_p\llbracket G \rrbracket $-module, and we can define the $\mu$-invariants in both settings. In the following, we prove relations between these invariants. Let $e$ and $f$ be the ramification index and the inertia degree of $p$ in the extension $F/\Q_p$. Let $M$ be a $\Ok\llbracket G \rrbracket $-module and consider the two elementary modules
\[\bigoplus_{i=1}^s \Ok\llbracket G \rrbracket /(\pi^{e_i}) ,\quad \bigoplus_{j=1}^t\Z_p\llbracket G \rrbracket /(p^{g_j})\]
associated to ${M[p^\infty] = M[\pi^\infty]}$. 
Note that $\bigoplus_{j=1}^t\Z_p\llbracket G \rrbracket /(p^{g_j})\cong \bigoplus_{i=1}^s \Ok\llbracket G \rrbracket /(\pi^{e_i})$ as $\Z_p\llbracket G \rrbracket $-modules. Indeed, we have an exact sequence
\[0\longrightarrow K_1\longrightarrow M[p^\infty]\longrightarrow \bigoplus_{i=1}^s \OkG/(\pi^{e_i})\longrightarrow K_2\longrightarrow 0,\]
where $K_1$ and $K_2$ are pseudo-null as $\Ok\llbracket G \rrbracket $-modules. We have to show that $K_1$ and $K_2$ are also pseudo-null as $\Z_p\llbracket G \rrbracket $-modules. As both $K_1$ and $K_2$ are annihilated by some power of $p$, it suffices to show that the $\mu$-invariants of $K_1$ and $K_2$ as  $\Z_p\llbracket G \rrbracket $-modules vanish. 

The quotient $\pi^lK_j/\pi^{l+1}K_j$ is a torsion $(\mathcal{O}/\pi)\llbracket G \rrbracket $-module for every ${l\in \N}$ and ${1 \le j \le 2}$ by \cite[Lemme~1.9]{perbet}. Note that $(\mathcal{O}/\pi)\llbracket G \rrbracket $ is finitely generated as a $\F_p\llbracket G \rrbracket $-module. In particular, $\pi^lK_j/\pi^{l+1}K_j$ is also a torsion $\F_p\llbracket G \rrbracket $-module. Considering successively $$l = ie, ie + 1, ie + 2, \ldots, ie + e-1,$$ 
we may conclude that  $p^iK_j/p^{i+1}K_j$ is a torsion $\F_p\llbracket G \rrbracket $-module for every ${i \in \N}$. Since ${\mu_{\Z_p\llbracket G \rrbracket }(K_j) = \sum_{i \ge 0} \textup{rank}_{\F_p\llbracket G \rrbracket }(p^i K_j / p^{i+1} K_j)}$, we obtain that the $\mu$-invariant of $K_j$ vanishes as a $\Z_p\llbracket G \rrbracket $-module.
\begin{lemma}
  We have \[f\sum_{i=1}^s e_i=\sum_{j=1}^t g_j. \]
\end{lemma}
\begin{proof}
It suffices to prove the claim for the case that $s=1$. We will proceed by induction over $e_1$. Assume that $e_1=1$. Then 
\[\Ok\llbracket G \rrbracket /(\pi) \cong \mathbb{F}_{p^f}\llbracket G \rrbracket \cong (\mathbb{F}_p\llbracket G \rrbracket )^f, \]
where the last isomorphism is one of $\mathbb{F}_p\llbracket G \rrbracket $-modules. So we have established the claim in this case.
Now let us assume that we have already proved the claim for the module $\Ok\llbracket G \rrbracket /(\pi^n)$, i.e. that ${\mu_{\Z_p \llbracket G \rrbracket}(\OkG/(\pi^n)) = f \cdot n}$. Consider the tautological exact sequence
\[0\longrightarrow \pi^n\Ok\llbracket G \rrbracket /\pi^{n+1}\Ok\llbracket G \rrbracket \longrightarrow \Ok\llbracket G \rrbracket /\pi^{n+1}\Ok\llbracket G \rrbracket \longrightarrow \Ok\llbracket G \rrbracket /\pi^n\longrightarrow 0.\]
As $\mu$-invariants are well behaved under exact sequences by \cite[Corollary 3.37]{venjakob}, the claim follows once we can show that $\mu_{\Z_p\llbracket G \rrbracket }(\pi^n\Ok\llbracket G \rrbracket /\pi^{n+1}\Ok\llbracket G \rrbracket )=f$. There is a canonical isomorphism $\pi^n\Ok\llbracket G \rrbracket /\pi^{n+1}\Ok\llbracket G \rrbracket  \cong \Ok\llbracket G \rrbracket /\pi\Ok\llbracket G \rrbracket $ and the claim is immediate from the case $e_1=1$.
\end{proof}
\begin{cor} \label{cor:mu-skalierung-f} 
   If $M$ is a finitely generated $\OkG$-module, then 
   \[ \mu_{\Z_p \llbracket G \rrbracket }(M) = f \cdot \mu_{\OkG}(M). \] 
\end{cor} 
\begin{proof} 
  In view of \cite[Proposition~3.34]{venjakob}, it suffices to consider elementary $\Z_p$-torsion $\OkG$-modules. Now apply the previous lemma. 
\end{proof} 

A second property that we will need frequently is the behavior of $\mu$-invariants under tensor products. Let $M$ be a $\Z_p\llbracket G \rrbracket $-module. Then $M\otimes \Ok$ is a $\Ok\llbracket G \rrbracket $-module.
\begin{lemma} \label{lemma:mu-skalierung-e} 
  \[ \mu_{\Z_p\llbracket G \rrbracket}(M)= \frac{1}{e} \cdot \mu_{\OkG}(M\otimes_{\Z_p} \Ok). \]
\end{lemma} 
\begin{proof}
As $\mu$-invariants are well-behaved under pseudo-isomorphisms and exact sequences of torsion $\OkG$-modules by \cite[Proposition~3.34 and Corollary~3.37]{venjakob}, it suffices to prove the lemma for a module $M=\Z_p\llbracket G \rrbracket /(p^n)$. If we tensor with $\mathcal{O}$, we obtain $\OkG/(p^n)=\OkG/(\pi^{en})$, from which the claim is immediate. 
\end{proof} 

\subsection{The $l_0$-invariant} 
In the special case $G\cong \Z_p^l$, ${l \in \N}$, one can define further invariants attached to finitely generated $\OkG$-modules. To this purpose, we note that in this case every finitely generated $\OkG$-module $M$ is pseudo-isomorphic to an elementary $\OkG$-module of the form 
\[ E_M = T(M) \oplus \bigoplus_{i = 1}^s \OkG/(\pi^{n_i}) \oplus \bigoplus_{j = 1}^t \OkG/(f_j), \] 
where ${T(M) \subseteq M}$ denotes the maximal torsion-free quotient, and ${f_1, \ldots, f_t \in \OkG}$ are power series which are not divisible by $\pi$ (see \cite[Proposition~5.1.7]{nsw}). An analogous property holds in the special case ${\Ok = \Z_p}$, with $\pi$ replaced by $p$. Note that ${\mu(M) = \mu_{\OkG}(M) = \sum_i n_i}$. 

Now we define the \emph{$l_0$-invariant} of $M$, as a $\OkG$-module, as follows. Write ${f_M = \pi^{\mu(M)} \cdot \prod_{j=1}^t f_j}$ (this is the so-called \emph{characteristic power series} of $M$), and let ${g_M = f_M / \pi^{\mu(M)}}$. Then we consider the (non-zero) image $\overline{g_M}$ of $g_M$ in the quotient ring ${\Omega\llbracket G\rrbracket := \OkG/(\pi)}$, and we let 
\[ l_0(M) = (l_0)_{\OkG}(M) = \sum_{\mathcal{P}}v_{\mathcal{P}}(\overline{g_M}), \] 
where the sum runs over the prime ideals $\mathcal{P}$ of $\Omega\llbracket G\rrbracket $ of the form ${\mathcal{P} = (\overline{\sigma - 1})}$, for elements ${\sigma \in G \setminus G^p}$. This is a finite sum since $\Omega\llbracket G\rrbracket $ is a Noetherian domain. 

Similarly, considering $M$ as a $\Z_p\llbracket G\rrbracket $-module, we have a characteristic power series $f_M^{\Z_p \llbracket G\rrbracket } \in \Z_p\llbracket G \rrbracket $ attached to $M$, and we let $g_M^{\Z_p\llbracket G\rrbracket } = f_M^{\Z_p\llbracket G\rrbracket } / p^{\mu_{\Z_p\llbracket G\rrbracket }(M)}$. Then we define 
\[ (l_0)_{\Z_p\llbracket G\rrbracket }(M) = \sum_{\mathcal{P}} v_{\mathcal{P}}(\overline{g_M^{\Z_p\llbracket G\rrbracket }}), \] 
where the sum now runs over the prime ideals ${\mathcal{P} = (\overline{\sigma - 1})}$ of the quotient algebra ${\Z_p\llbracket G\rrbracket /(p)}$. 

We would like to understand the relations between $(l_0)_{\OkG}$ and $(l_0)_{\Z_p\llbracket G\rrbracket }$. As a first step we will show that every pseudo-null $\Ok\llbracket G\rrbracket $-module is also pseudo-null as a $\Z_p\llbracket G\rrbracket $-module. Note that a $\Ok\llbracket G\rrbracket $-module $M$ is pseudo-null in the case ${G \cong \Z_p^l}$ if and only if the annihilator of $M$ contains two coprime elements $f_1$ and $f_2$. Using the Galois module structure behind the finite local extension $\mathcal{O}/\Z_p$ of rings, there is a well-defined norm map $N\colon \mathcal{O}\llbracket G\rrbracket \longrightarrow \Z_p\llbracket G\rrbracket $. Let $Q_1, \dots, Q_r$ be a complete list of all prime elements such that $N(Q_i)$ is not coprime with $N(f_1)$. Let $1, \dots, r_1$ be the indices such that $Q_i\mid f_2$, let $r_1+1, \dots, r_2$ be the indices such that $Q_i\mid f_1$ and let $r_2+1,\dots, r$ be the indices such that $Q_i$ divides neither $f_1$ nor $f_2$. Note that these three sets of indices are disjoint as $f_1$ and $f_2$ are relatively prime. Choose now a power series $g$ that is coprime to $Q_i$ for $1\le i\le r_2$ and divisible by $Q_i$ for $r_2+1\le i\le r$. Then none of the $Q_i$ divides $f_2+gf_1$.

Thus, by substituting $f_2$ by  $f_2+gf_1$ we can always assume that $N(f_1)$ and $N(f_2)$ are coprime. Note that $N(f_1)$ and $N(f_2)$ lie inside the intersection $\textup{Ann}(M)\cap \Z_p\llbracket G\rrbracket $. Therefore $M$ is also pseudo-null as a $\Z_p\llbracket G\rrbracket $-module. So if we want to understand how $l_0$-invariants behave under ring extensions, it suffices to consider elementary modules.
\begin{lemma} \label{lemma:l0-1} Let $M$ be a finitely generated $\Z_p\llbracket G\rrbracket $-module. Then
  \[(l_0)_{\Ok\llbracket G\rrbracket }(M\otimes_{\Z_p} \Ok)=(l_0)_{\Z_p\llbracket G\rrbracket }(M). \]
\end{lemma}
\begin{proof}
  It suffices to consider a module of the form $M=\Z_p\llbracket G\rrbracket /(h)$. Then $M\otimes \Ok\cong \Ok\llbracket G\rrbracket /(h)$ and the claim is immediate. 
\end{proof}
\begin{lemma} \label{lemma:l0-2} 
  Let $M$ be a finitely generated $\Ok\llbracket G\rrbracket $-module. Then
  \[ef(l_0)_{\Ok\llbracket G\rrbracket }(M)=(l_0)_{\Z_p\llbracket G\rrbracket }(M). \]
\end{lemma}
\begin{proof}
We can assume that $M=\Ok\llbracket G\rrbracket /(h)$ and that $h=g_0^t$ is the power of an irreducible element. Let $P$ be a generator of the prime ideal $\Z_p\llbracket G\rrbracket \cap \Ok\llbracket G\rrbracket g_0$. Then we obtain a decomposition $P=g_0^{e_0}\prod_{i=1}^{s-1}g_i^{e_i}$, where the $g_i$ generate prime ideals in $\Ok\llbracket G\rrbracket $. In the case that $F/\Q_p$ is Galois, all the $e_i$ are equal to each other and we can find elements $\sigma_i\in \Gal(F/\Q_p)$ such that $\sigma_i(g_0)=g_i$ for all $1\le i\le s-1$. For each $g_i$ we can define a decomposition group $G_i=\{\sigma\in \Gal(F/\Q_p)\mid (\sigma(g_i))=(g_i)\}$. The easiest case will be that all the $G_i$ are equal to each other. So we will start with the case that $F/\Q_p$ is an abelian extension. In this case there are integers $n$ and $q$, with $q$ coprime to $p$, such that $\Ok\subseteq \Z_p[\zeta_{qp^n}]$.

We start with the case that all $e_i=1$ and that $s=[\Ok:\Z_p]$, i.e $\sigma(g_0)$ and $g_0$ are coprime for all $\sigma\in\Gal(F/\Q_p)$.
Note that $\sigma(\tau-1)=\tau-1$ for every $\tau\in G$ and $\sigma\in \Gal(F/\Q_p)$. As the maximal ideal of $\Ok$ is invariant under $\Gal(F/\Q_p)$ as well, we see that the modules $\Ok\llbracket G\rrbracket/(g_i^t)$ all have the same $l_0$-invariants, both over $\Ok\llbracket G\rrbracket$ and $\Z_p\llbracket G\rrbracket$. Further, there is a pseudo-isomorphism
\[\Ok\llbracket G\rrbracket/(P^t) \longrightarrow \bigoplus_{i=0}^{s-1} \Ok\llbracket G\rrbracket/(g_i^t).\]
Computing the $l_0$-invariant over $\Z_p\llbracket G\rrbracket$ on both sides gives\[(l_0)_{\Z_p\llbracket G\rrbracket}(\Ok\llbracket G\rrbracket/(P^t))=s(l_0)_{\Z_p\llbracket G\rrbracket}(\Ok\llbracket G\rrbracket/ (g_0^t)) .\]
Similarly we obtain
\[(l_0)_{\Ok\llbracket G\rrbracket}(\Ok\llbracket G\rrbracket/(P^t))=s(l_0)_{\Ok\llbracket G\rrbracket}(\Ok\llbracket G\rrbracket/ (g_0^t)) .\]
As $P\in \Z_p\llbracket G\rrbracket$, it is a direct consequence of the definitions that 
\[(l_0)_{\Ok\llbracket G\rrbracket}(\Ok\llbracket G\rrbracket/(P^t))=(l_0)_{\Z_p\llbracket G\rrbracket}(\Z_p\llbracket G\rrbracket/(P^t)).\]
Using that $\Ok\llbracket G\rrbracket/(P^t)\cong (\Z_p\llbracket G\rrbracket/P^t)^{ef}$ as $\Z_p\llbracket G\rrbracket$-module it follows that 
\[ef(l_0)_{\Ok\llbracket G\rrbracket}(\Ok\llbracket G\rrbracket/(P^t))=(l_0)_{\Z_p\llbracket G\rrbracket}(\Ok\llbracket G\rrbracket/(P^t)),\] 
establishing the claim in this case. 

We now consider the case $F\subseteq \Q_p(\zeta_{qp^n})$ without any additional assumptions.
Let $G'\subseteq \Gal(F/\Q_p)$ be the subgroup that fixes $(g_0)$ as an ideal and let $\Ok'=\Ok^{G'}$. 
Using the fact that $l_0$-invariants are additive along short exact sequences we can use 
\[0\longrightarrow g_0 \Ok\llbracket G\rrbracket /(g_0^t) \Ok\llbracket G\rrbracket \longrightarrow \Ok\llbracket G\rrbracket /(g_0^t) \longrightarrow \Ok\llbracket G\rrbracket /(g_0) \longrightarrow 0\] 
to see that \[(l_0)_{\Z_p\llbracket G\rrbracket}(\Ok\llbracket G\rrbracket/(g_0^t)) = t(l_0)_{\Z_p\llbracket G\rrbracket}(\Ok\llbracket G\rrbracket/(g_0))\] 
for every $t\ge 1$. Let $s$ be minimal such that $(g_0^s)$ is defined already over $\Ok'$. 
Then there is a power series $h'\in\Ok'\llbracket G\rrbracket$ such that $(h')=(g_0^s)$. Since 
\[(l_0)_{\Ok\llbracket G\rrbracket }(\Ok\llbracket G\rrbracket/(h'))=[\Ok:\Ok'](l_0)_{\Ok'\llbracket G \rrbracket }(\Ok\llbracket G\rrbracket /(h')), \] 
a combination of the two previous cases concludes the analysis of the abelian case. 

Let now $\Ok$ be arbitrary and let $\Ok'$ be the maximal cyclotomic subring. We call a prime ideal $(\psi)$ with a generator $\psi\in \Z_p\llbracket G\rrbracket$ special if it generates the ideal of power series vanishing on a $(l-1)$-dimensional $\Z_p$-flat, see \cite[Section~3]{cuoco-monsky}. We call an ideal $(\psi)\subseteq \Ok\llbracket G\rrbracket$ special if $(\psi)$ divides a special prime $(\psi')$ with $\psi'\in \Z_p\llbracket G\rrbracket$. Note that the rings $\Ok\llbracket G\rrbracket $ and $\Ok'\llbracket G\rrbracket $ have the same special primes.

Let $M'\subseteq M$ be the unique maximal submodule that is annihilated by a finite product of special primes. Let $F'$ be the product of all special primes that divide $h$, then $M'=(h/F')M$ is again cyclic and annihilated by $F'$. We get a short exact sequence of $\Ok\llbracket G\rrbracket $-modules
\[0\longrightarrow M'\longrightarrow M\longrightarrow M/M'\longrightarrow 0.\]
For the remaining proof let $N$ denote the norm from $\Ok$ to $\Ok'$. It induces a well defined map $\Ok\llbracket G\rrbracket\longrightarrow \Ok'\llbracket G\rrbracket$.
As a $\Ok\llbracket G\rrbracket $-module the characteristic ideal of $M/M'$ is not divisible by any special prime. By definition $h/F'$ is coprime to every special prime and annihilates $M''$. Thus, $N(h/F')$ is not divisible by any special prime. On the other hand $N(h/F')\in (h/F')\Ok\llbracket G\rrbracket \cap \Z_p\llbracket G\rrbracket $. Thus, the characteristic ideal of $M/M'$ as a $\Z_p\llbracket G\rrbracket $-module is coprime to every special prime. As $l_0$-invariants are additive, it suffices to prove the lemma for $M'$ and $M/M'$ separately. 

Let us first deal with the case of $M'$. By construction $M'\cong \Ok\llbracket G\rrbracket /(F')$. But ${F'\in \Ok'\llbracket G\rrbracket }$, so this case can be established as above. Let us now consider ${M''=M/M'}$, and write ${F'' = h/F'}$.  
Let $W$ be the group of $p$-power roots of unity in $\mathbb{C}_p$ and denote by $W[p^n]$ all elements of order at most $p^n$.
Recall from \cite{cuoco-monsky} that we have a well-defined homomorphism
\[\psi\colon \Z_p\llbracket G\rrbracket \longrightarrow \bigoplus \Z_p[\zeta],\]
where the sum runs over all equivalence classes of $(W^l)[p^n]$ modulo Galois equivalence. This homomorphism is defined as follows. Write an element ${\zeta \in (W^l)[p^n]}$ as ${\zeta = (\zeta_1, \ldots, \zeta_l)}$. Then the component of $\psi(F)$ in $\Z_p[\zeta]$ is defined as 
\[ \psi(F)_\zeta = F(\zeta_1-1, \ldots, \zeta_l - 1) \] 
for each ${F \in \Z_p\llbracket G\rrbracket }$. 

We can define a $\Ok\llbracket G\rrbracket $-version of this by
\[\psi'\colon \Ok\llbracket G\rrbracket \longrightarrow \bigoplus \Ok[\zeta].\]
Note that the right hand side is a finitely generated free $\Ok$-module.
Let $I_n$ be the kernel of this map; note that $I_n$ depends on the concrete choice of representatives of $\zeta$ modulo Galois equivalence. The cokernel of $\psi$ is annihilated by $p^{nl}$ by \cite[Lemma 2.1]{cuoco-monsky}. Consider the ideal $(I_n,F'')$. $I_n$ contains the polynomials $\omega_n(T_i)$ that vanish on $(W^l)[p^n]$. As $F''$ is not divisible by any special primes, it follows that $F''$ is coprime to $(\omega_n(T_1),\ldots, \omega_n(T_l))$. We therefore obtain that 
\[\textup{ht}((F'',I_n))>\textup{ht}(I_n)=l.\] 
In particular, $\rg_{\Z_p}(M''/I_n)=O(p^{n(l-2)})$.
Using \cite[Lemma 2.4]{cuoco-monsky}, we can conclude that 
\[v_p(\vert( M''/I_nM'')[p^\infty]\vert)=v_p(\vert (\bigoplus \Ok[\zeta]/\psi'(F''))[p^\infty]\vert)+O(p^{n(l-1)}). \]
Note that $\Ok[\zeta]/(\zeta-1)$ has $p^{[\Ok:\Ok']}$ elements. Thus, $\Ok[\zeta]/\psi'(F'')[p^\infty]$ has 
\[ p^{[\Ok:\Ok']v_{\zeta-1}(F''(\zeta-1))}\] 
elements. Taking the sum over the equivalence classes of $\zeta$ gives
\[v_p(\vert( M''/I_nM'')[p^\infty]\vert)=[\Ok:\Ok']\sum_\zeta v_{\zeta-1}(F''(\zeta-1))+O(p^{n(l-1)}),\]
where the sum runs over all elements in $(W^l)[p^n]$. Analogously to \cite[Theorem 1.7]{cuoco-monsky}\footnote{The results Cuoco and Monsky use in their proofs are all stated for general coefficient rings and not only for $\Z_p$.} we obtain that 
\[ [\Ok:\Ok']\sum_\zeta v_{\zeta-1}(F''(\zeta-1))=[\Ok:\Ok'](l_0)_{\Ok\llbracket G\rrbracket }(F'')np^{n(l-1)}+O(p^{n(l-1)}).\]

We can also generalise \cite[Theorem 3.4]{cuoco-monsky} (in the same way) to obtain that
\[v_p(\vert( M''/I_nM'')[p^\infty]\vert)=(l_0)_{\Ok'\llbracket G\rrbracket }np^{n(l-1)}+O(p^{n(l-1)}). \]
Comparing the last two equations gives the desired relation between $\Ok\llbracket G\rrbracket $- and $\Ok'\llbracket G\rrbracket $-$l_0$ invariants.

Using the cyclotomic argument for the step from $\Ok'$ to $\Z_p$ gives the claim.
\end{proof}

\subsection{Uniform groups}
It is well-known (see Lemma~\ref{lemma:lazard} below) that each compact $p$-adic Lie group contains a \emph{uniform} normal subgroup of finite index, by which we mean the following. 
\begin{def1} 
  Let $G$ be a finitely generated pro-$p$-group. We write $G_n$, $n \in \N$, for the subgroups in the lower $p$-series of $G$ (see \cite[Definition~1.15]{Dixon}; $G_n$ corresponds to $P_{n+1}(G)$ in the notation from \cite{Dixon}). 
  
  Then $G$ is called \emph{uniform of dimension $l$}, ${l \in \N}$, if $G$ is profinite, topologically generated by $l$ generators, and there exists a filtration 
  \[ G = G_0 \supseteq G_1 \supseteq \ldots \supseteq G_n \supseteq \ldots \] 
  such that each $G_{n+1}$ is normal in $G_n$ and ${G_n/G_{n+1} \cong (\Z/p\Z)^l}$ for every ${n \in \N}$. 
\end{def1} 
If $G$ is uniform of dimension $l$, then ${G_n = G^{p^n}}$ (see \cite[Theorem~3.6]{Dixon}) and therefore ${[G : G_n] = p^{nl}}$ for every ${n \in \N}$. Moreover, each such group $G$ satisfies ${G[p^\infty]= \{0\}}$ by \cite[Theorem~4.5]{Dixon}. 
\begin{lemma} \label{lemma:lazard} 
  Each compact $p$-adic Lie group contains an open normal uniform subgroup. 
\end{lemma} 
\begin{proof} 
  This is a classical result of Lazard, see \cite[Corollary~8.34]{Dixon}. 
\end{proof} 
Therefore we will usually reduce to the case where the Galois group 
\[ G = \Gal(K_\infty/K)\] 
of an admissible Lie extension $K_\infty/K$ is a uniform group, using the following 
\begin{lemma} \label{lemma:uniform} 
  Let $G$ be a compact $p$-adic Lie group, and let ${G_0 \subseteq G}$ be a uniform normal subgroup of finite index. We write ${[G : G_0] = x}$. Let $M$ be a finitely generated $\OkG$-module. Then the following statements hold. \begin{compactenum}[(a)] 
     \item $M$ is also finitely generated as a $\Ok\llbracket G_0\rrbracket $-module, and 
     \[ \rg_{\Ok\llbracket G_0\rrbracket }(M) = x \cdot \rg_{\OkG}(M). \] 
     \item $\mu^{(k)}_{\Ok\llbracket G_0\rrbracket }(M) = x \cdot \mu^{(k)}_{\OkG}(M)$ for every ${k \in \N}$. 
  \end{compactenum}
\end{lemma} 
\begin{proof} 
  The first assertion follows since $\OkG$ is a free $\Ok\llbracket G_0\rrbracket $-module of rank $x$. Concerning the second statement, we just note that the truncated $\mu$-invariants of two pseudo-isomorphic $\OkG$-modules are equal by \cite[Proposition~3.34]{venjakob}. It therefore suffices to look at the corresponding elementary $\OkG$-modules, and the assertion follows since $\F_q\llbracket G\rrbracket $ is a free $\F_q\llbracket G_0\rrbracket $-module of rank $x$. 
\end{proof} 

\subsection{Ideal class groups} \label{section:ideal-class-groups} 
For a number field $L$, we denote by $\textup{Cl}(L)$ the ideal class group of $L$, and we let ${Y(L) \subseteq \textup{Cl}(L)}$ be its $p$-primary subgroup. More generally, if $L$ is any (possibly infinite) algebraic extension of a number field $K$, then we define 
\[ \textup{Cl}(L) = \varprojlim_{K \subseteq L' \subseteq L} \textup{Cl}(L') \quad \text{ and } \quad Y(L) = \varprojlim_{K \subseteq L' \subseteq L}Y(L'), \] 
where $L'$ runs over the finite subextensions of $L$ containing $K$. Here the projective limits are taken with respect to the natural norm maps. 

If $L$ is a number field and $\Sigma$ is a finite set of primes of $L$, then we denote by $\textup{Cl}_\Sigma(L)$ and $Y_\Sigma(L)$ the quotients of $\textup{Cl}(L)$ and $Y(L)$ by the subgroups which are generated by the primes in $\Sigma$. For a (possibly infinite) algebraic extension $L$ of a number field $K$ and a finite set $\Sigma$ of primes of $K$, we let 
\[ \textup{Cl}(L) = \varprojlim_{K \subseteq L' \subseteq L} \textup{Cl}_{\Sigma(L')}(L') \quad \text{ and } \quad Y(L) = \varprojlim_{K \subseteq L' \subseteq L} Y_{\Sigma(L')}(L'). \] 
By class field theory, one has a natural isomorphism 
\[ Y_\Sigma(L) \cong \Gal(H_\Sigma(L)/L), \] 
where $H_\Sigma(L)$ denotes the maximal unramified abelian pro-$p$-extension of $L$ in which each prime in $\Sigma(L)$ is totally split. 

If $L/K$ is normal with Galois group $G$, then $Y(L) \otimes_{\Z_p} \Ok$ is an $\OkG$-module.

\subsection{Fine Selmer groups of $p$-adic Galois representations} \label{section:fine-Selmer-groups} 
For any discrete $\Z_p$-module $M$, we denote the \emph{Pontryagin dual} of $M$ by 
\[ M^\vee = \textup{Hom}_{\textup{cont}}(M, \Q_p/\Z_p). \] 

If $L$ is any number field and $v$ is a prime of $L$, then we denote by $L_v$ the completion of $L$ at $v$. Fixing an algebraic closure $\overline{L}$ of $L$, we denote by ${G_L = \Gal(\overline{L}/L)}$ the absolute Galois group. For any $G_L$-module $M$, we let ${H^i(L,M) = H^i(G_L,M)}$ denote the corresponding cohomology groups, ${i \in \N}$. Moreover, if $L'/L$ is a normal algebraic extension, then we abbreviate $H^i(\Gal(L'/L), M)$ to $H^i(L'/L,M)$ for any $\Gal(L'/L)$-module $M$.  

Now fix a number field $K$, and let $V$ be a $p$-adic Galois representation of $G_K$. If ${p = 2}$, then we assume $K$ to be totally imaginary. As in Section~\ref{subsection:general}, we fix a lattice $T$ in $V$ and we let ${A = V/T}$. For any algebraic extension $L$ of $K$, we denote by $A(L)$ the maximal submodule of $A$ on which $\Gal(\overline{K}/L)$ acts trivially. Moreover, for any prime $v$ of $L$, we denote by ${A(L_v) \subseteq A}$ the maximal submodule on which the local absolute Galois group $G_{L_v}$ acts trivially (here $G_{L_v}$ embeds into the absolute Galois group $G_L$ canonically). 

Following \cite{coates-sujatha}, we now define the ($\pi$-primary part of the) \emph{fine Selmer group} of $A$ over $K$ as 
\[ \Sel_{0,A}(K) = \ker \left( H^1(K,A) \longrightarrow \prod_v H^1(K_v, A) \right). \] 
It is often better to work with the following definition, which is equivalent to the one given above under a certain natural hypothesis (which we will describe below). 

We fix a finite set $\Sigma$ of primes of $K$ which contains $\Sigma_p$ and $\Sigma_{\textup{ram}}(A)$, and we let $K_\Sigma$ be the maximal algebraic extension of $K$ which is unramified outside of $\Sigma$. Then we define 
\[ \Sel_{0,A, \Sigma}(K) = \ker \left( H^1(K_\Sigma/K, A) \longrightarrow \prod_{v \in \Sigma} H^1(K_v, A) \right). \] 
For any (not necessarily finite) extension ${L \subseteq K_\Sigma}$ of $K$, we let 
\[ \Sel_{0,A,\Sigma}(L) = \varinjlim_{K \subseteq L' \subseteq L} \Sel_{0,A, \Sigma(L')}(L'), \] 
where ${L' \subseteq L}$ runs over the finite subextensions. Note that ${K_\Sigma = L'_{\Sigma(L')}}$ for each such $L'$, since ${L \subseteq K_\Sigma}$. Therefore each $\Sel_{0, A, \Sigma(L')}(L')$ is contained in $H^1(K_\Sigma/K,A)$. 

It follows from results of Sujatha and Witte (see \cite[Section~3]{sujatha-witte}) that the second definition becomes independent of the choice of $\Sigma$ as soon as $L$ contains the cyclotomic $\Z_p$-extension $K_\infty^c$ of $K$; in fact, in this situation we have 
\[ \Sel_{0,A, \Sigma}(L) = \Sel_{0,A}(L)\] 
for each finite set $\Sigma$ containing $\Sigma_p$ and $\Sigma_{\textup{ram}}(A)$. Going through the arguments from \cite[Section~3]{sujatha-witte}, one derives that more generally, the following result holds: 
\begin{lemma} \label{lemma:sujatha-witte} 
  Using the above notation, suppose that ${L \subseteq K_\Sigma}$ contains a $\Z_p$-extension ${K_\infty}$ of $K$ such that no prime ${v \in \Sigma}$ is completely split in $K_\infty/K$. Then 
  \[ \Sel_{0,A, \Sigma}(L) = \Sel_{0,A}(L)\] 
  does not depend on the choice of the set $\Sigma$. 
\end{lemma} 

Now fix a set $\Sigma$ as above, and let $K_\infty/K$ be an admissible $p$-adic Lie extension in the sense of Section~\ref{subsection:Iwasawa-modules}. We suppose that ${\Sigma_{\textup{ram}}(K_\infty/K) \subseteq \Sigma}$. Then we have defined above the fine Selmer group 
\[ \Sel_{0,A, \Sigma}(K_\infty) = \varinjlim_{K \subseteq L' \subseteq K_\infty} \Sel_{0,A, \Sigma(L')}(L') \] 
(here again $L'$ runs over the finite subextensions of $K_\infty$), and we let 
\[ Y_{A, \Sigma}^{(L')} = \Sel_{0,A, \Sigma(L')}(L')^\vee \] 
be the corresponding Pontryagin duals. Finally, the projective limit 
\[ Y_{A, \Sigma}^{(K_\infty)} = \varprojlim_{K \subseteq L' \subseteq K_\infty} Y_{A, \Sigma}^{(L')} = \Sel_{0,A, \Sigma}(K_\infty)^\vee \] 
is taken with respect to the corestriction maps. 

It is well-known that $Y_{A, \Sigma}^{(K_\infty)}$ is a finitely generated $\Ok\llbracket G\rrbracket $-module, where we let ${G = \Gal(K_\infty/K)}$. 

Finally, for any $k \in \N$, we define \emph{$k$-truncated fine Selmer groups} as follows. For the number field $K$ and a finite set $\Sigma$ of primes of $K$, we let 
\[ \Sel_{0, A[\pi^k], \Sigma}(K) = \ker \left( H^1(K, A[\pi^k]) \longrightarrow \prod_{v \in \Sigma} H^1(K_v, A[\pi^k])\right). \] 
Similarly, if ${L \subseteq K_\Sigma}$ is an algebraic extension of $K$, then we let 
\[ \Sel_{0, A[\pi^k], \Sigma}(L) = \varinjlim_{K \subseteq L' \subseteq L} \Sel_{0, A[\pi^k], \Sigma(L')}(L'), \] 
where as usual $L'$ runs over the finite subextensions of $L$. 

\subsection{Modular forms} 
Let $f$ be a normalised newform of level $N$, weight $k$ and Nebentypus $\chi$. We denote by $L(f,s)=\sum_{n\in \N}a_nn^{-s}$ the $L$-function associated to $f$. Let $F^{(f)}$ be the smallest extension of $\Q$ that contains all the coefficients $a_n$. Let $v$ be a prime above $p$ in $F^{(f)}$ and denote by ${F = F_v}$ the completion of $F^{(f)}$ at $v$. Deligne associated a two-dimensional $F_v$-representation of $G_\Q$ to this modular form. We denote the representation space by $V_v$. We will always choose a Galois invariant lattice $T_f\subseteq V_f$ and denote the quotient $V_f/T_f$ by $A_f$. Let $\mathcal{O}$ be the ring of integers of ${F = F_v}$, as usual. Then $A_f$ is isomorphic to $(F_v/\mathcal{O})^2$ as $\mathcal{O}$-module. 

In the case that $f$ is the modular form attached to an elliptic curve defined over $\Q$, the field $F_v$ is isomorphic to $\Q_p$ and we can choose the lattice $T_f$ such that $A_f(1)\cong E[p^\infty]$, where $(1)$ denotes the Tate-twist. In analogy to elliptic curves with complex multiplication we define: 
\begin{def1} \label{def:CM-Form} 
We say that $f$ has complex multiplication if there exist an imaginary quadratic field $K$ and a  Hecke character $\psi$ of $K$ with infinity type\footnote{Here we use the same convention of infinity type as Kato in \cite{kato}.} $(1-k,0)$ such that
\[L(f,s)=L(\psi,s).\]
\end{def1}
Note that the conductor of $\psi$ divides the level $N$ of $f$.
Let $L$ be the smallest field containing $K$ and all the values of $\psi$. It is immediate that $F^{(f)} \subseteq L$. We fix once and for all a prime $w$ above $v$ in $L$. We denote by $V_{L_w}(\psi)$ a one-dimensional $V_{L_w}$-vector space on which $\Gal(\overline{K}/K)$ acts via $\psi$. Let $\mathfrak{f}$ be the conductor of $\psi$. Then the Galois action of ${G_K = \Gal(\overline{K}/K)}$ on $V_{L_w}(\psi)$ factors through the ray class field $K(\mathfrak{f}p^\infty)$. We can turn $V_{L_w}(\psi)$ into a $\Gal(\overline{K}/\Q)$-representation by choosing a lift $\iota$ of the complex conjugation to $\Gal(\overline{K}/\Q)$ and looking at the representation space 
\[V_{L_w}(\psi)\oplus \iota V_{L_w}(\psi). \]
An element $\sigma\in \Gal(\overline{K}/\Q)$ sends a tuple $(x,\iota y)$ to $(\sigma x, \iota (\iota \sigma\iota)y)$, an element of the form $\iota \sigma$ sends $(x,\iota y)$ to $((\iota \sigma\iota)y, \iota \sigma x)$. Let $\Tilde{K}$ be a finite extension of $K$. From the definition it is obvious that $A_f^{\Gal(\overline{K}/\Tilde{K})}$ is finite. Furthermore, as $\Gal(\overline{K}/\Q)$-module $A_f$ is irreducible.
 
By \cite[Lemma 15.11]{kato} we see that
\[V_f\otimes L_w\cong V_{L_w}(\psi)\oplus \iota V_{L_w}(\psi).\]
\begin{lemma}
\label{lemm:ground-field-cm}
  Let $f$ be a CM form in the sense of Definition~\ref{def:CM-Form}. Then we can choose a finite extension $K'$ of $K$ such that the $\Gal(\overline{\Q}/K')$-representation $V_f\otimes K'_w$ (here $w$ is a suitable place above $p$ in $K'$) is unramified outside of the primes above $p$. Furthermore we can choose $K'$ such that $K'(A_f)$ is a $\Z_p^2$-extension and such that $A_f[p]$ is defined over $K'$.
\end{lemma}
\begin{proof}
  Note that $\Gal(K(\mathfrak{f}p^\infty)/K)\cong \Delta \times \Z_p^2$ for some finite abelian group $\Delta$. Thus, $L(A_f)\subseteq LK(\mathfrak{f}\overline{\mathfrak{f}}p^\infty)$. Note that $K(\mathfrak{f}\overline{\mathfrak{f}}p^\infty)/K(\mathfrak{f}\overline{\mathfrak{f}}p)$ is a $\Z_p^2$-extension ( in the case that $p=2$ one might have to look at $K(\mathfrak{f}\overline{\mathfrak{f}}p^\infty)/K(\mathfrak{f}\overline{\mathfrak{f}}p^2)$) and that this extension is unramified outside of $p$. Let $K'=LK(\mathfrak{f}\overline{\mathfrak{f}}p)$. Then $\psi$ factors through $\Gal(LK(\mathfrak{f}p^\infty)/K')$. As $K'/K(\mathfrak{f}p)$ is a finite extension, it follows that $LK(\mathfrak{f}p^\infty)/K'$ is a $\Z_p^2$-extension and $K'$ satisfies both conditions of the lemma.
\end{proof}
In particular, the trivialising extension $K'(A_f)$ contains the cyclotomic $\Z_p$-extension of $K'$ since ${A_f[p^\infty] \subseteq K'(A_f)}$, and therefore Lemma~\ref{lemma:sujatha-witte} implies that the fine Selmer group of $A_f$ over $K'(A_f)$ does not depend on the choice of the finite set $\Sigma$. 

\subsection{Rankin-Selberg convolutions}
Let $f_1, \ldots, f_k$ be normalised newforms with corresponding Galois representations ${V_{f_1}, \ldots, V_{f_k}}$. As before we fix Galois stable lattices ${T_{f_1}, \ldots, T_{f_k}}$, respectively. Assume that all representations are defined over the same field $L$. Then the Rankin-Selberg L-function $L(f_1 \times \ldots \times f_k, s)$ is also defined over the field $L$ (see \cite[Theorem~1.6.3]{bump}), and we can consider the $2k$-dimensional representation ${V_{f_1}\otimes \ldots \otimes V_{f_k}}$ with the $\Ok$-stable lattice ${T_{f_1} \otimes \ldots \otimes T_{f_k}}$. An easy computation in tensor products shows that 
\[(V_{f_1} \otimes \ldots \otimes V_{f_k})/(T_{f_1} \otimes \ldots \otimes T_{f_k})\cong V_{f_1}/T_{f_1}\otimes \ldots \otimes V_{f_k}/T_{f_k}=A_{f_1} \otimes \ldots \otimes A_{f_k}.\]

\begin{lemma} \label{lemma:spielweise} 
  Assume that $f_1, \ldots, f_k$ are $CM$-forms with respect to the imaginary quadratic fields ${K_1, \ldots, K_k}$. Then there is a finite extension $K'$ of ${K_1 \cdot \ldots \cdot K_k}$  such that the $\Gal(\overline{\Q}/K')$-representation ${V_{f_1} \otimes \ldots \otimes V_{f_k} \otimes K'_w}$ (here $w$ is a suitable place above $p$ in $K'$) is unramified outside of the primes above $p$. 
  
  Furthermore, we can choose $K'$ such that ${K'(A_{f_1} \otimes \ldots \otimes A_{f_k})}$ is contained in a $\Z_p^{k+1}$-extension of $K'$ and such that $(A_{f_1}\otimes \ldots \otimes A_{f_k})[p]$ is defined over $K'$. 
  
  If the imaginary quadratic number fields $K_1, \ldots, K_k$ are linearly independent (i.e. the Galois group over $\Q$ of their compositum has $2$-rank equal to $k$), then $K'(A_{f_1} \otimes \ldots \otimes A_{f_k})/K'$ is a $\Z_p^{k+1}$-extension. 
\end{lemma} 
\begin{proof} 
Choose $K_j'$, ${j \in \{1, \ldots, k\}}$, such that these fields satisfy the conclusion of Lemma \ref{lemm:ground-field-cm} for the modular forms ${f_1, \ldots, f_k}$, respectively. Take $K'$ to be the compositum of ${K'_1, \ldots, K_k'}$. Then ${K'(A_{f_1}\otimes \ldots \otimes A_{f_k})}$ is contained in the compositum of the $\Z_p^2$-extensions of ${K_1, \ldots, K_k}$. This is at most a $\Z_p^{k+1}$-extension since the cyclotomic $\Z_p$-extension is contained in each of the $\Z_p^2$-extensions. 

Now suppose that $K_1, \ldots, K_k$ are linearly disjoint. We denote the  corresponding \emph{anti-cyclotomic} $\Z_p$-extensions by $K_1^a, \ldots, K_k^{a}$. For any ${i \in \{1, \ldots, k\}}$ and any $\Z_p$-extension $L/K_i$ we call the $\Z_p$-extension ${L \cdot K' / K'}$ the \emph{shift} of $L$ to $K'$. Now we use an auxiliary 
\begin{lemma} 
  For each ${i \in \{1, \ldots, k\}}$, the shift of the anti-cyclotomic $\Z_p$-extension $K_i^a$ is not contained in the composite of the shifts of the $K_j^a$, $j \ne i$. 
\end{lemma} 
\begin{proof} 
  Fix ${i \in \{1, \ldots, k\}}$. Let $q \ne p$ be a prime number which splits in $K_i/\Q$ and is inert in $K_j$ for every ${j \ne i}$ (the existence of infinitely many such primes follows from the Chebotarev density theorem). Then $q$ will be totally split in the anticyclotomic $\Z_p$-extension $K_j^a$ of $K_j$ for every $j \ne i$ (see \cite[Chapter~13, Theorem~5.2]{lang_cyclotomic}), and therefore $q$ is totally split in the compositum of all these $\Z_p$-extensions over $K'$. On the other hand, the two primes of $K_i$ above $q$ will split into finitely many primes of $K_i^a$, by \cite[Theorem~2]{brink}. This shows that the shift of the anti-cyclotomic $\Z_p$-extension $K_i^a$ cannot be contained in the above compositum. 
\end{proof} 
Now we return to the proof of Lemma~\ref{lemma:spielweise}. Using an inductive argument, it is easy to conclude that the composite of the shifts of the $\Z_p^2$-extensions of $K_1, \ldots, K_k$ is a $\Z_p^{k+1}$-extension of $K'$ (note that the cyclotomic and the anticyclotomic $\Z_p$-extensions are the only $\Z_p$-extensions of $K_i$ which are normal over $\Q$, and consider the eigenspaces of complex conjugation on the Galois group of the maximal abelian $p$-ramified pro-$p$-extension of $K'$). 
\end{proof}

\section{Auxiliary results} \label{section:auxiliary-results} 

The following result on the structure of Iwasawa modules is taken from \cite{kleine-mueller2}. Recall the definition of the truncated $\mu$-invariants $\mu^{(k)}$ from Section~\ref{subsection:Iwasawa-modules}. 
\begin{thm} \label{thm:iwasawamodules} 
   Let $G$ be a uniform $p$-adic Lie group of dimension $l$. Let $M$ be a finitely generated $\OkG$-module, and let ${k \in \N}$. For any ${n \in \N}$, we denote by $M_{G_n}$ the maximal quotient of $M$ on which ${G_n = G^{p^n}}$ acts trivially. 
   
   Then 
   \[ |v_p(|M_{G_n}/\pi^k|) - (\rg_{\OkG}(M) \cdot f k + \mu^{(k)}(M) \cdot f ) \cdot p^{nl}| = O(k p^{n(l-1)}). \] 
   Here the implicit $O$-constant does not depend on $k$. 
\end{thm} 
\begin{proof} 
  An analogous statement has been proved in  \cite[Theorem~5.5]{kleine-mueller2} for the case ${F = \Q_p}$, i.e. ${\Ok = \Z_p}$ (the corresponding result itself is a generalisation of \cite[Theorem~2.1(ii)]{perbet} and was proved along the same lines). The arguments carry over to the more general setting with only minor changes, and therefore we provide only a rather brief sketch. 
  
  For every finitely generated $\OkG$-module $N$, the homology groups 
  \[ H_i(G_n, N) = \textup{Tor}_i^{\OkG}(\Ok, N) \] 
  can be bounded as follows: for every ${i \in \N}$, we have 
  \[ v_p(|H_i(G_n,N)|) = O(p^{n(l-1)}), \] 
  where $l = \dim(G)$ (this is a direct analogue of \cite[Corollaire~2.3]{perbet}). 
  
  If $N$ is a finitely generated $\OkG$-module such that $N/\pi$ is a torsion $\F_q\llbracket G \rrbracket $-module, then the proof of \cite[Corollaire~2.4]{perbet} (cf. also the proof of \cite[Theorem~5.5]{kleine-mueller2}) implies that 
  \[ v_p(|H_i(G_n, N/\pi^k)|) = O(k p^{n(l-1)}), \] 
  where the implicit constant does not depend on $k$. 
  
  One can deduce from this fact (see \cite[Proposition~2.5]{perbet} and \cite[Proposition~5.6]{kleine-mueller2}) that for any finitely generated $\OkG$-modules $M$ and $N$ which are pseudo-isomorphic, one has 
  \[ |v_p(|M_{G_n}/\pi^k|) - v_p(|N_{G_n}/\pi^k|)| = O(k p^{n(l-1)}). \] 
  By the general structure theory of finitely generated $\OkG$-modules,\footnote{Although the results in \cite{coates-schneider-sujatha} are proved only for $\Z_p\llbracket G\rrbracket $-modules, the conclusions remain valid if $\Z_p$ is replaced by $\Ok$.}  we are reduced to consider elementary $\OkG$-modules. In fact, as in \cite[Proposition~5.6]{perbet}, we are reduced to consider the maximal $\OkG$-torsion submodule $t(M)$ and the quotient ${\tilde{M} = M/t(M)}$ separately. 
  
  Concerning the torsion part, it follows from \cite[Lemme~2.8]{perbet} that 
  \[| v_p(|t(M)_{G_n}/\pi^k|) - \mu^{(k)} \cdot f p^{ln} | = O(k p^{n(l-1)}); \] 
  the reason for the appearance of the factor $f$ is that ${v_p(|\Ok/(\pi^i)|) = fi}$ for every ${i \in \N}$. 
  It remains to study the growth of the coinvariants of ${\tilde{M} = M/t(M)}$. Since the quotient $\tilde{M}/\pi$ is annihilated by $p$, \cite[Theorem~2.1(i)]{perbet} implies that 
  \begin{eqnarray*} 
     v_p(|(\tilde{M}/\pi)_{G_n}|) = v_p(|\tilde{M}_{G_n}/\pi|) & = & \rg_{\F_p \llbracket G \rrbracket }(\tilde{M}/\pi) \cdot p^{nl} + O(p^{n(l-1)}) \\ 
     & = & f \cdot \rg_{\F_q\llbracket G\rrbracket }(\tilde{M}/\pi) \cdot p^{nl} + O(p^{n(l-1)}). 
  \end{eqnarray*} 
  Since $\tilde{M}$ is $\OkG$-torsionfree by construction, it follows from \cite[Lemma~2.2]{kleine-mueller1} that 
  \[ \rg_{\F_q\llbracket G \rrbracket }(\tilde{M}/\pi) = \rg_{\OkG}(\tilde{M}), \] 
  and the latter is of course the same as $\rg_{\OkG}(M)$. Therefore the proof of \cite[Lemme~2.9]{perbet} implies that 
  \[ |v_p(|\tilde{M}_{G_n}/\pi^k|) - f k \cdot \rg_{\OkG}(M)\cdot p^{nl}| = O(k p^{n(l-1)}). \] 
\end{proof} 

\begin{lemma}
\label{lem:cm-torsion contained}
  Let $f$ be a $CM$ modular form and let $K'$ be the field defined in Lemma~\ref{lemm:ground-field-cm}. Let $K_\infty=K'(A_f)$. Then $A_f[p^n]\subseteq A_f(K_n)$. 
\end{lemma}
\begin{proof}
We prove the claim by induction: by the hypothesis on $K'$, we see that $A_f[p]$ is defined over $K'\subseteq K_1$. Assume that $A_f[p^n]\subseteq A_f(K_n)$, and let 
\[ x\in A_f[p^{n+1}]\setminus A_f[p^n]. \] 
Then $px=y$ is invariant under $\Gal(K_\infty/K_n)$. Let $\sigma\in \Gal(K_\infty/K_n)$, then ${\sigma(x)=x+v}$ for some ${v\in A_f[p]}$. Thus, ${\sigma^p(x)=x+pv=x}$, showing that the subextension $K_n(A_f[p^{n+1}])/K_n$ is of exponent $p$. As $A_f$ is defined over $K_\infty$, we obtain that ${A_f[p^{n+1}]\subseteq A_f(K_{n+1}})$.
\end{proof}

\begin{lemma} \label{lemma_lim-murty} 
   Let $A$ be associated with a $p$-adic representation of $G_K$ of dimension $d$, and let $K_\infty/K$ be an admissible $p$-adic Lie extension with Galois group $G$. Let $\Sigma$ be a finite set of primes of $K$ which contains $\Sigma_p$, $\Sigma_{\textup{ram}}(A)$ and $\Sigma_{\textup{ram}}(K_\infty/K)$. If ${p = 2}$, then we assume that $K$ is totally imaginary. 
   
   Let ${k \in \N}$. If ${L \subseteq K_\Sigma}$ is a finite extension of $K$ such that ${A[\pi^k] \subseteq A(L)}$, then 
   \[ |v_p(|\Sel_{0,A[\pi^k], \Sigma(L)}(L)|) - d \cdot v_p(|(Y(L) \otimes_{\Z_p} \Ok)/\pi^k|)| \le f d k |\Sigma(L)|, \] 
   where $p^f = |\Ok/(\pi)|$ denotes the inertia degree of $p$ in the extension $F/\Q_p$. 
\end{lemma} 
This lemma is a generalisation of \cite[Theorem~5.1]{lim-murty}, and the proof is essentially the same: 
\begin{proof} 
  Recall that $K_\Sigma$ denotes the maximal $\Sigma$-ramified algebraic extension of $K$. Since ${A[\pi^k] \subseteq A(L)}$ by hypothesis, we have 
  \[ H^1(K_\Sigma/L, A[\pi^k]) = \textup{Hom}(K_\Sigma/L, A[\pi^k]), \] 
  and 
  \[ H^1(L_v, A[\pi^k]) = \textup{Hom}(L_v, A[\pi^k])\] 
  for each prime $v$. 
  
  Therefore 
  \[ \Sel_{0, A[\pi^k], \Sigma(L)}(L) \cong \textup{Hom}(Y_\Sigma(L), A[\pi^k]). \] 
  The homomorphisms on the right hand side can be described as follows. Suppose that we look for homomorphisms 
  \[ \phi \colon \Z/p^x \Z \longrightarrow (F/\Ok)[\pi^k],  \] 
  for some ${x \in \N}$. 
  Then $\phi$ is completely determined by the image of a generator $b$ of $\Z/p^x \Z$. Now $\phi(b) = a \cdot \pi^{-y}$ for some ${a \in \Ok/(\pi^y)}$ and ${y = \min(ex, k)}$. 
  Therefore one has exactly $p^{f \cdot \min(ex,k)}$ possibilities for $\phi$. 
  
  On the other hand, ${(\Z/p^x \Z \otimes \Ok)/\pi^k \cong \Ok/\pi^{\min(ex,k)}}$ has the same cardinality. Since ${A \cong (F/\Ok)^d}$, this shows that the order of $\textup{Hom}(Y_\Sigma(L), A[\pi^k])$ is equal to $p^z$ with 
  \[ z = d \cdot v_p(|(Y_\Sigma(L) \otimes_{\Z_p} \Ok)/\pi^k|). \] 
  It therefore remains to compare $(Y(L) \otimes_{\Z_p} \Ok)/\pi^k$ and $(Y_\Sigma(L) \otimes_{\Z_p} \Ok)/\pi^k$. To this purpose, we note that the canonical surjection ${\textup{Cl}(L) \twoheadrightarrow \textup{Cl}_\Sigma(L)}$ induces an exact sequence 
  \[ C \otimes_{\Z_p} \Ok \longrightarrow Y(L) \otimes_{\Z_p} \Ok \longrightarrow Y_\Sigma(L) \otimes_{\Z_p} \Ok \longrightarrow 0, \] 
  where $C$ is an abelian group of $p$-rank at most $|\Sigma(L)|$. Therefore we obtain an exact sequence 
  \[ (C \otimes_{\Z_p} \Ok)/\pi^k \longrightarrow (Y(L) \otimes_{\Z_p} \Ok)/ \pi^k \longrightarrow (Y_\Sigma(L) \otimes_{\Z_p} \Ok)/ \pi^k \longrightarrow 0. \] 
  This proves the lemma. 
\end{proof}

\section{Control theorems} \label{section:control} 

\begin{lemma} \label{lemma:control-alt} 
  Let $A$ be associated with a $p$-adic representation of $G_K$ of dimension $d$, and let $\Sigma$ be a finite set of primes of $K$ containing $\Sigma_p$ and $\Sigma_{\textup{ram}}(A)$. If ${p = 2}$, then we assume that $K$ is totally imaginary. 
  
  Let $L$ be a finite extension of $K$ which is unramified outside of $\Sigma$. Then 
  \[ |v_p(|\Sel_{0,A, \Sigma(L)}(L)[\pi^k]|) - v_p(|\Sel_{0,A[\pi^k], \Sigma(L)}(L)|)| \le f d k (1 + |\Sigma(L)|)\] 
  for each ${k \in \N}$, where $f$ denotes the inertia degree of $p$ in the extension $F/\Q_p$. 
\end{lemma} 
\begin{proof} 
  This is \cite[Lemma~3.1]{kleine-mueller1}. 
\end{proof} 

\begin{thm} \label{control-allg} 
  Let $A$ be associated with a $p$-adic representation of $G_K$ of dimension $d$, let $K_\infty/K$ be an admissible $p$-adic Lie extension of dimension $l$, and write ${G = \Gal(K_\infty/K)}$. We fix a finite set $\Sigma$ of primes of $K$ which contains $\Sigma_p$, $\Sigma_{\textup{ram}}(A)$ and $\Sigma_{\textup{ram}}(K_\infty/K)$. If ${p = 2}$, then we assume that $K$ is totally imaginary. 
  
  Let $k \in \N$. If the decomposition subgroup ${D_v \subseteq G}$ is infinite for each ${v \in \Sigma}$, then the kernel and the cokernel of the natural map 
  \[ r_n \colon \Sel_{0, A[\pi^k], \Sigma(K_n)}(K_n) \longrightarrow \Sel_{0, A[\pi^k], \Sigma}(K_\infty)^{G_n} \] 
  are finite. Moreover, $v_p(|\ker(r_n)|) = O(k)$ and $v_p(|\coker(r_n)|) = O(k p^{n(l-1)})$. Here the implicit constants do not depend on $k$ or $n$. 
\end{thm} 
\begin{proof} 
We start from the following commutative diagram:  
  \[\begin{tikzcd}[scale cd = 0.85] 
0\arrow[r]&\textup{Sel}_{0,A[\pi^k],\Sigma(K_n)}(K_n)\arrow[r]\arrow[d,"r_n"]&H^1(K_{\Sigma(K_n)}/K_n,A[\pi^k])\arrow[r]\arrow[d,"h_n"]& \prod_{v\in \Sigma(K_n)}H^1({K_{n,v}},A[\pi^k])\arrow[d,"g_n"] \\0 \arrow[r]&\textup{Sel}_{0,A[\pi^k], \Sigma}(K_\infty)^{G_n}\arrow[r]&H^1(K_{\Sigma}/K_\infty,A[\pi^k])^{G_n}\arrow[r]& \prod_{v\in \Sigma(K_\infty)}H^1({K_{\infty,v}},A[\pi^k]). \end{tikzcd}
\] 
In view of the snake lemma, we have to bound the kernels of $g_n$ and $h_n$ and the cokernel of $h_n$. Now the inflation-restriction exact sequence implies that 
\[ \ker(h_n) = H^1(G_n, A(K_\infty)[\pi^k]) \quad \text{ and } \quad \coker(h_n) = H^2(G_n, A(K_\infty)[\pi^k]), \] 
and a similar description exists for the kernel of $g_n$. Since ${A \cong (F/\Ok)^d}$, it follows that $A[\pi^k]$ is finite. Therefore \cite[Lemma~2.2]{kundu-lim} implies that the kernel and cokernel of $h_n$ are finite and of bounded order. 

The same holds for the kernel of each map 
\[ g_{n,v} \colon H^1(K_{n,v}, A[\pi^k]) \longrightarrow H^1(K_{\infty,v}, A[\pi^k]). \] 
Now we use the hypothesis on the decomposition groups. Let ${G_0 \subseteq G}$ be a uniform normal subgroup of finite index (cf. Lemma~\ref{lemma:lazard}). Let ${K' = K_\infty^{G_0}}$. Then $K'$ is a finite extension of $K$, and no prime ${v \in \Sigma(K')}$ is totally split in the uniform extension $K_\infty/K'$. Therefore \cite[Lemma~4.1]{limwildkernel} implies that $|\Sigma(K_n)| = O(p^{n(l-1)})$. This proves the desired bound for the order of $\ker(g_n)$. 
\end{proof} 

\subsection{Control theorems in the CM case}

Let $f$ be a CM modular form of weight $k \ge 2$ with Hecke character $\psi$ (over $K$). Then $\psi$ has infinity type $(1-k,0)$, see \cite[15.10]{kato}.  Let $L_w$ be the $p$-adic field containing all values of $\psi$. Then $\Gal(\overline{\Q}/K)$ acts on $L_w(\psi)$ as follows. Let $\mathfrak{f}$ be the conductor of $\psi$ and let $\mathfrak{a}$ be an ideal that is coprime to $p\mathfrak{f}$. Then the action of $\Gal(\overline{\Q}/K)$ on $L_w(\psi)$ factors through $\Gal(K(\mathfrak{f}p^\infty)/K)$, and the Artin symbol of $\mathfrak{a}$ in $\Gal(K(\mathfrak{f}p^\infty)/K)$ acts on $L_w(\psi)$ via $\psi(\mathfrak{a})^{-1}$.

\begin{lemma}
\label{lemma:finite-fixgroup}
Let $f$ be a modular form with complex multiplication by some Hecke character $\psi$ which is defined over $K$. Let $K'$ be a finite extension of $K$ and let $w$ be a prime above $p$ such that $V_f \otimes K_w'$ is as a $\Gal(\overline{\Q}/K')$-module unramified outside of the primes above $p$. Let $K_\infty/K'$ be a $\Z_p^l$-extension such that $A_f$ is defined over $K_\infty$. Let $\mathfrak{p}$ be a prime above $p$ in $K'$ and let $D_{\mathfrak{p}}$ be its decomposition group in $\Gal(K_\infty/K')$. 
Then $A_f^{D_\mathfrak{p}}$ is finite. 
\end{lemma}\begin{proof}
The representation $V_f$ factors through $\Gal(K(\mathfrak{f}'p^\infty)/K)$ for some suitable ideal $\mathfrak{f}'$. It therefore suffices to consider the case that $K_\infty/K'$ is a $\Z_p^2$-extension which contains the unique $\Z_p^2$-extension of $K$. Then $D_\mathfrak{p}$ is a finite index subgroup of $\Gal(K_\infty/K')$. Let $\mathfrak{a}$ be a prime ideal of $K'$ which is totally split in $K'/K$ and such that the Frobenius $\sigma_\mathfrak{a}$ of $\mathfrak{a}$ lies inside $D_\mathfrak{p}$. Let $t$ be the least positive integer such that $\mathfrak{a}^t=(\alpha)$ with $\alpha\equiv 1\pmod{\mathfrak{f}'}$. We obtain that $\sigma^t_\mathfrak{a}$ acts on $A_f$ via multiplication by $\psi^{-1}(\mathfrak{a}^t)=\alpha^{k-1}$ which is certainly non-trivial (i.e. $\ne 1$). Thus, the subgroup of $A_f$ that is fixed under $\sigma_{\mathfrak{a}}$ is finite, since it is contained in $A_f[\pi^k]$ for a suitable ${k \gg 0}$. As $\sigma_{\mathfrak{a}}\in D_{\mathfrak{p}}$, the claim follows.
\end{proof}

\begin{thm}
\label{control-cm}
  Let $f_1, \ldots, f_k$ be modular forms with complex multiplication by imaginary quadratic fields $K_1, \ldots, K_k$. Suppose that we are in one of the following cases 
  \begin{enumerate}[(i)]
      \item Suppose that ${k = 1}$. Let $K'$ be a finite extension of $K_1$ and let $w$ be a prime above $p$ such that $V_{f_1} \otimes K_w'$ is as a $\Gal(\overline{\Q}/K')$-module unramified outside of the primes above $p$. Let $K_\infty/K'$ be a $\Z_p^l$-extension such that $A_{f_1}$ is defined over $K_\infty$. 
  \item Suppose that the imaginary quadratic fields $K_i$ are linearly independent. Let $K'$ be an extension of ${K_1 \cdot \ldots \cdot K_k}$ such that ${V_{f_1} \otimes \ldots V_{f_k} \otimes K_w'}$ is unramified outside $p$ and let $K_\infty/K'$ be a $\Z_p^l$-extension such that $A_{f_1}\otimes \ldots \otimes A_{f_k}$ is defined over $K_\infty.$ For each prime $\p$ of $K_\infty$ above $p$, let $D_\mathfrak{p}$ be the decomposition group in $K_\infty/K$ (note that there are only finitely many different decomposition groups $D_\mathfrak{p}$ as $K_\infty/K'$ is abelian). 
  
  We assume that $(A_{f_1}\otimes \ldots \otimes A_{f_k})^{D_\mathfrak{p}}$ is finite for all $\mathfrak{p}$.
  \end{enumerate}
 Let ${G = \Gal(K_\infty/K')}$, ${G_n = G^{p^n}}$ and ${K_n = K_\infty^{G_n}}$ for each ${n \in \N}$. Then the natural map
  \[r_n\colon \Sel_{0,f}(K_n)\longrightarrow \Sel_{0,f}(K_\infty)^{G_n}\]
  has finite kernel and cokernel. Moreover, in both cases 
  \[ v_p(|ker(r_n)|) = O(n) \quad \text{ and } \quad v_p(|\coker(r_n)|) = O(np^{n(l-1)}). \] 
  \end{thm}
\begin{proof}
Note that for a $CM$ modular form the definition of the fine Selmer group does not depend on the set $\Sigma$ due to Lemma~\ref{lemma:sujatha-witte}. We write $M$ for $A_{f_1}$ and ${A_{f_1}\otimes \ldots \otimes A_{f_k}}$ respectively. Consider the following diagram 
  \[\begin{tikzcd}[scale cd = 0.85] 
0\arrow[r]&\textup{Sel}_{0,M}(K_n)\arrow[r]\arrow[d,"r_n"]&H^1(K_{\Sigma_p}/K_n,M)\arrow[r]\arrow[d,"h_n"]& \prod_{v\in \Sigma_p(K_n)}H^1({K_{n,v}},M)\arrow[d,"g_n"] \\0 \arrow[r]&\textup{Sel}_{0,M}(K_\infty)^{G_n}\arrow[r]&H^1(K_{\Sigma_p}/K_\infty,M)^{G_n}\arrow[r]& \prod_{v\in \Sigma_p(K_\infty)}H^1({K_{\infty,v}},M)\end{tikzcd}
\]
  
We start by analysing the kernel and cokernel of $h_n$. Recall that 
\[ V_{f_i}\otimes L_{w}\cong L_w(\psi_i)\oplus \iota L_w(\psi_i). \] 
It is an easy computation to conclude that $M(K')$ is finite. So we can use \cite[Lemma 6.2]{kleine-mueller2} to see that for both ${i = 1}$ and ${i = 2}$, the group $H^i(K_\infty/K_n,M(K_\infty))$ is finite. By \cite[Lemma 2.1]{kundu-lim}, the $p$-ranks are bounded independently of $n$.
Note that $M(K_\infty)^\vee$ is finitely generated over $\Z_p\llbracket G\rrbracket $. Using \cite[Lemma 2.1.1]{Lim-Liang} together with \cite[Proposition 1.9.1]{nsw} (see also \cite[Proof of Theorem~6.8]{kleine-mueller2}), we can bound the exponent of the torsion subgroup of both cohomology groups by $p^{ln+b}$ for some constant $b$. Therefore $v_p(|\ker(h_n)|)$ and $v_p(|\coker(h_n)|)$ are of size $O(n)$.

It remains to estimate the kernel of $g_n=\prod_{v\mid p}g_{n,v}$. Here $g_{n,v}$ denotes the corresponding local map at the prime $v$. As $K_\infty$ contains the unique $\Z_p^2$-extension of $K_1$ (in the case that $M=A_{f_1}\otimes \ldots \otimes A_{f_k}$ it contains all the $\Z_p^2$-extensions of the $K_i$), we see that each prime above $p$ has a decomposition group of dimension at least two. In the first case we can use Lemma \ref{lemma:finite-fixgroup} to see that $M(K_{n,v})$ is finite for every prime $v\mid p$. In the second case our assumptions assure that $M(K_{n,v})$ is finite. Thus, in both cases we can use the same arguments as for the kernel of $h_n$ to see that $v_p(|\ker(g_{n,v})|)=O(n)$. Due to the fact that each decomposition group has dimension at least two, there are at most $O(p^{n(l-2)})$ primes in $K_{n,v}$ above $p$. Thus $v_p(\vert \ker g_n\vert)=O(np^{n(l-2)})$.
\end{proof}

\section{Proofs of the main results} \label{section:main} 

\begin{thm} \label{thm:A} 
  Let $A$ be associated with a representation of $G_K$ of dimension $d$, and let $K_\infty/K$ be an admissible $p$-adic Lie extension with Galois group $G$. We fix a finite set $\Sigma$ which contains $\Sigma_p$, $\Sigma_{\textup{ram}}(A)$ and $\Sigma_{\textup{ram}}(K_\infty/K)$. If ${p = 2}$, then we assume that $K$ is totally imaginary. 
  
  Let ${k \in \N}$. If the decomposition subgroup ${D_v \subseteq G}$ is infinite for each ${v \in \Sigma}$ and if ${A[\pi^k] \subseteq A(K_\infty)}$, then
  \begin{align} \label{eq:thmA} k \cdot \rg_{\OkG}(Y_{A, \Sigma}^{(K_\infty)}) + \mu^{(k)}_{\OkG}(Y_{A, \Sigma}^{(K_\infty)}) = d \cdot \mu^{(k)}_{\OkG}(Y(K_\infty)\otimes \Ok).\end{align}  
\end{thm} 
\begin{proof} 
  In view of Lemma~\ref{lemma:lazard}, there exists a uniform subgroup ${G_0 \subseteq G}$ of finite index. Let ${K' = K_\infty^{G_0}}$ be the subfield of $K_\infty$ which is fixed by $G_0$. Since each decomposition subgroup $D_v$ is infinite by hypothesis, no prime of $\Sigma$ is completely split in the induced uniform extension $K_\infty/K'$. This is all what we need in the sequel. We therefore may without loss of generality assume that $G$ itself is uniform because Lemma~\ref{lemma:uniform} implies that changing $K$ to $K'$ just implies that each term in the equation~\eqref{eq:thmA} will be multiplied by the same factor ${x = [G: G_0]}$. 
  
  Since no prime ${v \in \Sigma}$ is totally split in $K_\infty/K$, it follows from \cite[Theorem~6.1]{ochi-venjakob} that the projective limit ${Y(K_\infty) = \varprojlim_n Y(K_n)}$ of the ideal class groups is a torsion $\Z_p\llbracket G \rrbracket $-module. Let $x$ be minimal such that $ex\ge k$. We have seen in the proof of \cite[Theorem~7.2]{kleine-mueller2} (see in particular equation~(16) in this proof) that the kernels and the cokernels of the natural maps 
  \[ \alpha_n\colon Y(K_\infty)_ {G_n}/p^x \longrightarrow Y(K_n)/p^x \] 
  are finite and of rank $O(p^{n(l-1)})$ as $n \to \infty$. If we tensor both sides with $\Ok$, the kernel and cokernel stay of the same order of magnitude $O(p^{n(l-1)})$ (in fact, the $p$-valuations of the cardinalities are multiplied by $f$). In the following, we will denote the induced map on the tensored quotients also by $\alpha_n$. Thus, for our fixed ${k \in \N}$, the order of the cokernel of
  \[\alpha'_n\colon(Y(K_\infty)_{G_n}\otimes \Ok)/\pi^k\longrightarrow  (Y(K_n)\otimes \Ok)/\pi^k\]
  stays bounded (in rank) by $O(p^{n(l-1)})$. Let $x$ be an element in the kernel of $\alpha'_n$ and let $y\in (Y(K_\infty)_{G_n}\otimes \Ok)/\pi^{ex}$ be a preimage of $x$ under the canonical map 
  \[ (Y(K_\infty)_{G_n} \otimes \Ok)/\pi^{ex} \longrightarrow (Y(K_\infty)_{G_n} \otimes \Ok)/\pi^k. \] 
  Then $\alpha_n(y)\in \pi^k(Y(K_n)\otimes \Ok)/\pi^{ex}$. But the kernel and cokernel of $\alpha_n$ are of order of magnitude $O(p^{n(l-1)})$. Thus, 
  $$ \alpha_n^{-1}(\pi^k(Y(K_n)\otimes \Ok)/\pi^{ex})/(\pi^k(Y(K_\infty)_{G_n}\otimes \Ok)/\pi^{ex})$$ is of this size. It is straight forward that $\ker(\alpha'_n)$ is the image of 
  $$\alpha_n^{-1}(\pi^k(Y(K_n)\otimes \Ok)/\pi^{ex})/(\pi^k(Y(K_\infty)_{G_n}\otimes \Ok)/\pi^{ex})$$ in $(Y(K_\infty)_{G_n}\otimes \Ok)/\pi^k$. In particular, the kernel of $\alpha'_n$ is also of order of magnitude $O(p^{n(l-1)})$. 
  
  Therefore we obtain that 
  \[ |v_p(|(Y(K_n)\otimes \Ok)/\pi^k|) - v_p(|(Y(K_\infty)_{G_n}\otimes \Ok)/\pi^k |)| = O(p^{n(l-1)}). \] 
  It follows from Theorem~\ref{thm:iwasawamodules} that  
  \[ |\mu_{\OkG}^{(k)}(Y(K_\infty)\otimes \Ok) \cdot f p^{nl} - v_p(|(Y(K_\infty)_{G_n}\otimes \Ok)/\pi^k|)| = O(p^{n(l-1)}) \] 
  and therefore also 
  \[ | \mu_{\OkG}^{(k)}(Y(K_\infty)\otimes \Ok) \cdot f p^{nl} - v_p(|(Y(K_n)\otimes \Ok)/\pi^k|)| = O(p^{n(l-1)}). \] 
  On the other hand, Theorem~\ref{thm:iwasawamodules}, applied to $Y_{A, \Sigma}^{(K_\infty)}$, implies that 
  \[ | (\rg_{\OkG}(Y_{A, \Sigma}^{(K_\infty)}) \cdot f k + \mu^{(k)}(Y_{A, \Sigma}^{(K_\infty)}) f) \cdot p^{nl} - v_p(|\Sel_{0, A, \Sigma}(K_\infty)^{G_n}[\pi^k]|) | = O(p^{n(l-1)}). \] 
  Now we apply Lemma~\ref{lemma_lim-murty}, Lemma~\ref{lemma:control-alt} and Theorem~\ref{control-allg} in order to conclude the proof of the theorem. 
\end{proof} 

In the special setting of CM modular forms, we can prove the following more precise comparison result.  
\begin{thm} \label{thm:B} 
  Let ${M = A_f}$ be either associated with a CM modular form or let $M=A_{f_1}\otimes\cdots \otimes  A_{f_k}$ be associated with a Rankin-Selberg convolution of $k$ CM modular forms. We assume that ${M[p] \subseteq M(K)}$. In the first case we set $k=1$. In the second case we assume that $M^{D_\mathfrak{p}}$ is finite for all $\mathfrak{p}$ lying above $p$. Let $K_\infty/K$ be as in Theorem~\ref{control-cm}, i.e. $K_\infty$ is a trivialising extension for $M$, and let $e$ and $f$ denote the ramification index and the inertia degree of $p$ in the extension $F_w/\Q_p$ (respectively, in the extension $(F_1F_2....F_k)_w/\Q_p$ in the second case). Then 
  \[\mu_{\Z_p\llbracket G\rrbracket }(Y_{M}^{(K_\infty)})=2^k fe \mu_{\Z_p\llbracket G\rrbracket }(Y(K_\infty)), \quad (l_0)_{\Z_p\llbracket G\rrbracket }(Y_{M}^{(K_\infty)})=2^k fe(l_0)_{\Z_p\llbracket G\rrbracket }(Y(K_\infty)), \] 
  and \[ \rg_{\Z_p\llbracket G\rrbracket }(Y_{M}^{(K_\infty)})=0. \] 
\end{thm}
Note: since the cyclotomic $\Z_p$-extension of $K$ is contained in $K_\infty$, the fine Selmer groups do not depend on the chosen finite set $\Sigma$ by Lemma~\ref{lemma:sujatha-witte}, as long as $\Sigma$ contains $\Sigma_p$, $\Sigma_{\textup{ram}}(A)$ and $\Sigma_{\textup{ram}}(K_\infty/K)$. 
\begin{proof} 
In the following proof, we always consider ${\Sigma = \Sigma_p}$. The proof is similar to the proof of \cite[Theorem 7.21]{kleine-mueller2} and we only sketch it here. Using Lemma~\ref{lem:cm-torsion contained} and the fact that each decomposition group of any ${v \in \Sigma_p}$ in $K_\infty/K$ has dimension at least two, we can conclude that
\[\vert v_p(\vert \textup{Sel}_{0,M[p^n],\Sigma_p}(K_n)/p^n\vert- 2^k ef v_p(\vert Y(K_n)/p^n\vert)\vert=O(p^{(l-1)n}).\] 
Note that we get the extra factor $ef$ here due to the fact that we compare $$\textup{Sel}_{0,M[p^n],\Sigma_p}(K_n)$$ 
with $\textup{Hom}(Y_{\Sigma_p}(K_n),M[p^n])$. Recall that as $\mathcal{O}$-module $M\cong (F_v/\mathcal{O}(F_v))^{2^k}$. As a $\Z_p$-module it is isomorphic to $(\Q_p/\Z_p)^{2^k ef}$.
By Lemma~\ref{lemma:control-alt} we see that we can substitute $\textup{Sel}_{0,M[p^n],\Sigma_p}(K_n)$ by $\textup{Sel}_{0,M,\Sigma_p}(K_n)[p^n]$ without affecting the error term. By the control theorem (Theorem~\ref{control-cm}) we obtain that
\[\vert v_p(\vert (Y_M)_{G_n}/p^n\vert)-v_p(\vert Y_M^{(K_n)}/p^n\vert)\vert =O(p^{(l-1)n}).\]
We can now use \cite[Theorem 5.8]{kleine-mueller2} together with \cite[Lemma 7.29 and the proof of 7.28]{kleine-mueller2} to prove the desired relations of Iwasawa invariants. 
\end{proof} 
This also proves Theorem~\ref{thm:intro-B}, in view of the following 
\begin{rem} \label{rem:OkG} 
Note that this result is formulated for $\Z_p\llbracket G \rrbracket $-invariants. Using Lemmas~\ref{lemma:l0-1} and \ref{lemma:l0-2} and dividing through $ef$, we obtain that the $l_0$-invariants transform as follows. 

\[(l_0)_{\Ok \llbracket G\rrbracket}(Y_{M})=2^k (l_0)_{\Ok\llbracket G\rrbracket}(Y(K_\infty)\otimes \Ok).\]   For $\mu$-invariants we obtain from Corollary~\ref{cor:mu-skalierung-f} and Lemma~\ref{lemma:mu-skalierung-e} that $$\mu_{\Z_p\llbracket G \rrbracket }(Y_{M})=f\mu_{\Ok\llbracket G \rrbracket }(Y_{M})$$ and $$e\mu_{\Z_p\llbracket G \rrbracket }(Y(K_\infty))=\mu_{\Ok\llbracket G \rrbracket }(Y(K_\infty)\otimes \Ok).$$ 
So in this case we indeed obtain that 
\[\mu_{\Ok\llbracket G \rrbracket }(Y_{M})=2^k  \mu_{\Ok\llbracket G \rrbracket }(Y(K_\infty)\otimes \Ok).\]
\end{rem}

\bibliography{references} 
\bibliographystyle{plain} 

\end{document}